\documentclass[11 pt]{amsart}

\usepackage{amsfonts}
\usepackage{amssymb}
\usepackage{amsmath}
\usepackage{latexsym}
\usepackage{mathrsfs}
\usepackage{amsthm}

\newtheorem{thrm}{Theorem}[section]
\newtheorem{lem}[thrm]{Lemma}
\newtheorem{cor}[thrm]{Corollary}
\newtheorem{prop}[thrm]{Proposition}
\newtheorem{conj}[thrm]{Conjecture}

\theoremstyle{definition}
\newtheorem{defn}[thrm]{Definition}
\newtheorem{exmple}[thrm]{Example}
\newtheorem{rmk}[thrm]{Remark}

\begin{document}

\newcommand{\Supp}{\mathrm{Supp}}

\title{A Cone Theorem for Nef Curves}
\author{Brian Lehmann}
\thanks{This material is based upon work supported under a National Science
Foundation Graduate Research Fellowship.}
\address{Department of Mathematics, University of Michigan \\
Ann Arbor, MI \, \, 48109}
\email{blehmann@umich.edu}

\begin{abstract}
Following ideas of V.~Batyrev,
we prove an analogue of the Cone Theorem for the closed cone of nef curves: an enlargement of the
cone of nef curves is the closure of the sum of a $K_{X}$-non-negative portion and countably many $K_{X}$-negative
coextremal rays.  An example shows that this enlargement is necessary.  We also describe the relationship between $K_{X}$-negative faces of this cone and the possible outcomes of the minimal model program.
\end{abstract}

\maketitle

\section{Introduction}

Suppose that $X$ is a uniruled projective variety.  The results of \cite{bchm06} enable us to run the minimal model program on $X$: after a series of transformations we obtain a birational model $X'$ of $X$ with the structure of a Mori fiber space.  These Mori fiber spaces play a key role in describing the birational geometry of $X$.  Since the steps of the minimal model program are not uniquely determined, it is important to study the set of all such fibrations.  As realized by Mori, this data can be concisely encoded using the cone of curves.  Our main goal is to analyze the outcomes of the minimal model program by proving a structure theorem for the cone of nef curves -- that is, the set of curve classes that have non-negative intersection with every effective divisor.

The most important result concerning the cone of nef curves was formulated in
\cite{bdpp04}.  Recall that an irreducible curve $C$ on a variety $X$ is called movable if it is a member
of a family of curves that dominates $X$.
In \cite{bdpp04} it is shown that the cone of nef curves is the closure
of the cone generated by classes of movable curves.
However, one might hope to obtain more specific results for
the $K_{X}$-negative portion of the cone.
The main conjecture in this direction was given by Batyrev.  For a projective variety $X$, let $\overline{NE}_{1}(X)$ denote the closed cone of effective curves of $X$ and let $\overline{NM}_{1}(X)$ denote the closed cone of nef curves.

\begin{conj}[\cite{batyrev89}, Conjecture 4.4] \label{batyrevconjecture}
Let $(X,\Delta)$ be a klt pair.  There are countably many
$(K_{X}+\Delta)$-negative movable curves $C_{i}$ such that
\begin{equation*}
\overline{NE}_{1}(X)_{K_{X} + \Delta \geq 0} + \overline{NM}_{1}(X) =
\overline{NE}_{1}(X)_{K_{X} + \Delta \geq 0} + \sum \mathbb{R}_{\geq 0} [C_{i}].
\end{equation*}
The rays $\mathbb{R}_{\geq 0}[C_{i}]$ only accumulate along the hyperplane $(K_{X}+\Delta)^{\perp}$.
\end{conj}

This conjectural description of $\overline{NM}_{1}(X)$ should be seen as the analogue
of the Cone Theorem for effective curves:

\begin{thrm}[\cite{kawamata84}, Theorem 4.5 and \cite{kollar84}, Theorem 1] \label{conetheorem}
Let $(X,\Delta)$ be a klt pair.  There are countably many
$(K_{X}+\Delta)$-negative rational curves $C_{i}$ such that
\begin{equation*}
\overline{NE}_{1}(X) =
\overline{NE}_{1}(X)_{K_{X} + \Delta \geq 0} + \sum \mathbb{R}_{\geq 0} [C_{i}].
\end{equation*}
The rays $\mathbb{R}_{\geq 0}[C_{i}]$ only accumulate along
the hyperplane $(K_{X}+ \Delta)^{\perp}$.
\end{thrm}

\cite{batyrev89} proves Conjecture \ref{batyrevconjecture} in the case where $\Delta=0$ and $X$ is a threefold with terminal singularities.  \cite{araujo05} fixes an error in his proof and gives a very clear framework for the general case.  Work of a similar flavor has been done in \cite{xie05} and \cite{barkowski07}.

The techniques of \cite{bchm06} allow for new progress toward Conjecture \ref{batyrevconjecture}.  The case  where $(X,\Delta)$ is log Fano of arbitrary dimension was settled in \cite{bchm06}.  Our main result is a version of Conjecture \ref{batyrevconjecture} that holds in general but gives slightly less information about the accumulation behavior of the rays.  For the sake of completeness we will work with dlt pairs rather than klt pairs.

\begin{thrm} \label{firsttheorem}
Let $(X,\Delta)$ be a dlt pair.  There are countably many
$(K_{X}+\Delta)$-negative movable curves $C_{i}$ such that
\begin{equation*}
\overline{NE}_{1}(X)_{K_{X} + \Delta \geq 0} + \overline{NM}_{1}(X) =
\overline{NE}_{1}(X)_{K_{X} + \Delta \geq 0} + \overline{\sum \mathbb{R}_{\geq 0} [C_{i}]}.
\end{equation*}
The rays $\mathbb{R}_{\geq 0}[C_{i}]$ only accumulate along
hyperplanes that support both $\overline{NM}_{1}(X)$ and $\overline{NE}_{1}(X)_{K_{X} + \Delta \geq 0}$.
\end{thrm}

A weaker statement was proved independently in \cite{araujo08}.  Araujo demonstrates a similar equality but with no restriction on accumulation points (and thus no claim about the local finiteness of the rays).

It seems likely that further progress on the minimal model program is needed to establish Conjecture \ref{batyrevconjecture} in full.  In Section \ref{accumulationsection} we will show that when $\Delta$ is big Conjecture \ref{batyrevconjecture} follows from the conjecture on termination of flips.  As shown in \cite{batyrev89} and \cite{araujo08}, Conjecture \ref{batyrevconjecture} can also be derived from the Borisov-Alexeev-Borisov conjecture concerning boundedness of log Fano varieties.

It is worth noting two important differences between Theorems \ref{conetheorem} and \ref{firsttheorem}.
First, Theorem \ref{firsttheorem} includes the term $\overline{NE}_{1}(X)_{K_{X}+\Delta \geq 0}$
where we might expect $\overline{NM}_{1}(X)_{K_{X}+\Delta \geq 0}$ by analogy with the Cone Theorem.
This enlargement is necessary:~in Example \ref{cutkoskyexample} we construct a threefold for which
\begin{equation*}
\overline{NM}_{1}(X) \neq
\overline{NM}_{1}(X)_{K_{X} + \Delta \geq 0} + \sum \mathbb{R}_{\geq 0} [C_{i}]
\end{equation*}
for any locally discrete countable collection of curves.  This example was previously used in
\cite{cutkosky86} to find a divisor with no rational Zariski decomposition.

Second, we do not know whether the movable curves in Theorem \ref{firsttheorem} can
be chosen to be rational.  This would follow from standard (and difficult)
conjectures about log Fano varieties.  If we are willing to replace the $C_{i}$ by
classes $\alpha_{i}$ that do not necessarily represent curves, we may choose the
$\alpha_{i}$ to be formal numerical pullbacks of rational curves on birational
models; see Remark \ref{rationalremark}.

Returning to our original question, we would like to relate the structure of $\overline{NM}_{1}(X)$ to
outcomes of the minimal model program.  In contrast to the situation for $\overline{NE}_{1}(X)$, it is not true that a $K_{X}$-negative extremal face $F$ of $\overline{NE}_{1}(X)_{K_{X} + \Delta \geq 0} + \overline{NM}_{1}(X)$
canonically determines a rational map $f: X \dashrightarrow Z$ constructed using the minimal model program.
However, it turns out that $F$ uniquely determines a birational equivalence class of maps.  We need to include a weak condition on the face $F$ to ensure that it avoids any accumulation points.

\begin{thrm} \label{contraction}
Let $(X,\Delta)$ be a dlt pair.  Suppose that $F$ is a $(K_{X} + \Delta)$-negative
extremal face of $\overline{NE}_{1}(X)_{K_{X} + \Delta \geq 0} + \overline{NM}_{1}(X)$.  Suppose
furthermore that there is some pseudo-effective divisor class $\beta$
such that $\beta^{\perp}$ contains $F$ but avoids $\overline{NE}_{1}(X)_{K_{X} + \Delta \geq 0}$.
Then there is a birational morphism $\psi: W \to X$
and a contraction $h: W \to Z$ such that:
\begin{enumerate}
\item Every movable curve $C$ on $W$ with $[\psi_{*}C] \in F$ is contracted by $h$.
\item For a general pair of points in a general fiber of $h$, there is a movable curve $C$
through the two points with $[\psi_{*}C] \in F$.
\end{enumerate}
These properties determine the pair $(W,h)$ up to birational equivalence.
In fact the map we construct satisfies a stronger property:
\begin{itemize}
\item[(3)] There is an open set $U \subset W$ such that a complete curve $C$ in $U$ is contracted by $h$ iff $[\psi_{*}C] \in F$ and for a general fiber $H$ of $h$ there is a completion $\overline{H}$ of $U \cap H$ so that the complement of $U \cap H$ has codimension at least $2$.  Furthermore $\psi$ is an isomorphism on $U$.
\end{itemize}
\end{thrm}

The general strategy of the proof of Theorem \ref{firsttheorem} is as follows.
Let $D$ be a divisor such that $D^{\perp}$ supports $\overline{NE}_{1}(X)_{K_{X} + \Delta \geq 0}
+ \overline{NM}_{1}(X)$ only along the $K_{X}$-negative part of the cone.  Since
$D$ is positive on $\overline{NE}_{1}(X)_{K_{X} + \Delta \geq 0}$, after rescaling we can write
$D = K_{X} + \Delta + A$ for some ample divisor $A$.  By the results
of \cite{bchm06}, we can run the minimal model program with scaling for $D$.
After a number of divisorial contractions and flips, we
obtain a birational contraction $\phi: X \dashrightarrow X'$ on which $K_{X'} + \phi_{*}(\Delta + A)$ is
nef.  Furthermore, $X'$ has a
Mori fiber space structure $g: X' \to Z$, where $K_{X'} + \phi_{*}(\Delta + A)$ vanishes on every curve
contracted by $g$.  Choose a movable curve $C$ on $X'$ contracted by $g$ and sufficiently
general.  Then the curve $\phi^{-1}(C)$ lies on the
boundary of $\overline{NM}_{1}(X)$ along $D^{\perp}$.

The main point is to show that finitely many curves will suffice.  In fact,
all we need to know is that we obtain only finitely many different Mori fiber spaces as we vary $D$.
Previous work has relied upon boundedness of log Fano varieties.  We will instead use
the finiteness of minimal models proved in \cite{bchm06}.  The unusual condition
on the accumulation of rays in Theorem \ref{firsttheorem} arises from the fact that we need
to perturb $K_{X} + \Delta$ by an ample divisor $A$ in order to use the techniques of \cite{bchm06}.

I would like to thank J.~M\textsuperscript{c}Kernan for his extensive advice and support
and C.~Araujo for helpful comments on an earlier version.

\section{Preliminaries}

Suppose $X$ is a normal projective variety.
We will let $N^{1}(X) = NS(X) \otimes \mathbb{R}$ denote the N\'eron-Severi
space of $\mathbb{R}$-Cartier $\mathbb{R}$-Weil divisors up to numerical equivalence and $N_{1}(X)$ the dual space of curves
up to numerical equivalence.
As we run the minimal model program we will need to be careful to distinguish between divisors
(or curves)
and their numerical equivalence classes.  If $D$ is a Cartier divisor,
we will denote the numerical class of $D$ by $[D] \in N^{1}(X)$, and similarly for curves.

We have already introduced the cone of effective curves $\overline{NE}_{1}(X)$ -- that is, the
closure of the cone generated by classes of irreducible
curves.  We will use $\overline{NE}^{1}(X)$ to denote the cone
of pseudo-effective divisors.  We define the cone of nef curves $\overline{NM}_{1}(X) \subset N_{1}(X)$ as
the cone dual to $\overline{NE}^{1}(X)$ under the intersection product.
Finally, given a cone $\sigma$ and an
element $K$ of the dual vector space, $\sigma_{K \geq 0}$ denotes the intersection
of $\sigma$ with the closed half-space on which $K$ is non-negative, and $\sigma_{K = 0}$
denotes the intersection $\sigma \cap K^{\perp}$.

\begin{lem} \label{conepushforward}
Let $\pi: Y \to X$ be a birational morphism of normal projective varieties.
Then $\pi_{*}\overline{NE}_{1}(Y) = \overline{NE}_{1}(X)$ and
$\pi_{*}\overline{NM}_{1}(Y) = \overline{NM}_{1}(X)$.
\end{lem}

\begin{proof}
Recall that $\pi_{*}$ on curves is dual to the pull-back map $\pi^{*}: N^{1}(X) \to N^{1}(Y)$.  A divisor class $\alpha \in N^{1}(X)$ is pseudo-effective iff $\pi^{*}\alpha$ is pseudo-effective.  By duality we find that $\pi_{*}$ maps $\overline{NM}_{1}(Y)$ to $\overline{NM}_{1}(X)$.  If the map were not surjective, we could find an $\alpha \in N^{1}(X)$ that is not pseudo-effective but is positive on all of $\pi_{*}\overline{NM}_{1}(Y)$.  But then $\pi^{*}\alpha$ would be pseudo-effective, a contradiction.  A similar argument works for $\overline{NE}_{1}(X)$.
\end{proof}

If $C$ is an irreducible member of a family of curves dominating $X$,
we will say that $C$ is a \emph{movable} curve.  Every movable curve
$C$ is nef, but not conversely.

The standard definitions and theorems of the minimal model program
may be found in \cite{km98}.  However, in our definition of a pair $(X,\Delta)$ we
allow $K_{X} + \Delta$ to be $\mathbb{R}$-Cartier, not just $\mathbb{Q}$-Cartier.
This distinction makes little difference in the proofs.

We will also use one other non-standard piece of terminology.  Suppose that $(X,\Delta)$ is a
$\mathbb{Q}$-factorial klt pair and $\phi: X \dashrightarrow X'$ is a composition of $K_{X} + \Delta$
flips and divisorial contractions.  Let $\phi_{i}: X \dashrightarrow X_{i}$ denote the result
after $i$ steps and $R_{i}$ the extremal ray on $X_{i}$ that defines the $(i+1)$ step.  Given a divisor
$D$ on $X$, we say that $\phi$ is \emph{$D$-non-negative} if
$\phi_{i*}D \cdot R_{i} \leq 0$ for every $i$.

We will use the following lemma many times, often without explicit mention:

\begin{lem}[\cite{lazarsfeld04}, Example 9.2.29] \label{ampleqlinearklt}
Let $(X,\Delta)$ be a klt pair, and suppose that $H$ is an ample $\mathbb{R}$-Cartier
divisor.  There is some effective $H'$ that is $\mathbb{R}$-linearly equivalent to $H$ such that
$(X,\Delta + H')$ is klt.
\end{lem}

\section{Running the Minimal Model Program}

As explained in the introduction, we will use the minimal model program to show
that certain rays in $N_{1}(X)$ are generated by movable curves.
In this section we will extract the application of the minimal model program
in the form of Proposition \ref{mmpuse}.

\begin{lem} \label{rescaleampleklt}
Suppose that $(X, \Delta)$ is a klt pair and that $\{ H_{j} \}_{j=1}^{m}$ is a
finite collection of ample divisors.  Consider the set of divisors
\begin{equation*}
\mathcal{H} = \left\{ \left. \, \sum_{j=1}^{m} c_{j}H_{j} \, \right|  \, \,
0 \leq c_{j} \leq 1 \, \, \, \forall j \, \right\}.
\end{equation*}
There are finitely many effective ample divisors $\{W_{j}\}_{j=1}^{m}$ such that every
element of $\mathcal{H}$ is $\mathbb{R}$-linearly equivalent to a linear
combination $\sum_{j=1}^{m} a_{j}W_{j}$ where
\begin{equation*}
\left( X , \Delta + \sum_{j=1}^{m} a_{j}W_{j} \right)
\end{equation*}
is a klt pair.
\end{lem}

\begin{proof}
If $(X,\Delta)$ is a klt pair and $H$ is an ample divisor, then Lemma
\ref{ampleqlinearklt} guarantees the existence of an effective divisor $W$, $\mathbb{R}$-linearly
equivalent to $H$, such that $(X,\Delta + W)$ is a klt pair.  Using this fact inductively, we find
effective $W_{j}$, $\mathbb{R}$-linearly equivalent to $H_{j}$, such that $(X, \Delta + \sum W_{j})$ is klt.
Finally, recall that if $(X,\Delta + D)$ is any klt pair, then so is $(X,\Delta + D')$ for
every effective $D' \leq D$.  Thus any sum of the $W_{j}$ with smaller coefficients still yields
a klt pair.
\end{proof}

In order to prove local finiteness of the rays in Theorem \ref{firsttheorem}, we need to use
finiteness of ample models as proved in \cite{bchm06}.  The particular result we need is the following:
\begin{thrm}[\cite{bchm06}, Corollaries 1.1.5 and 1.3.2] \label{mmptheorem}

Let $(X,\Delta)$ be a $\mathbb{Q}$-factorial klt pair and let $A$ be an ample divisor on $X$.
Suppose $V$ is a finite dimensional subspace of the space
of real Weil divisors.  Define
\begin{equation*}
\mathcal{E}_{A} = \{ \, \, \Gamma \in V \, \, | \, \, \Gamma \geq 0 \textrm{ and }
K_{X} + \Delta + A + \Gamma \textrm{ is klt and pseudo-effective } \}.
\end{equation*}

For every $\Gamma \in \mathcal{E}_{A}$, we can run the minimal model program for
$D := K_{X} + \Delta + A + \Gamma$.  If $D$ lies on the boundary of the pseudo-effective cone,
this will result in a birational contraction $\phi_{i}: X \dashrightarrow X_{i}$ with a Mori fiber space structure
$g_{i}: X_{i} \to Z_{i}$, where $K_{X_{i}} + \phi_{i*}(\Delta + A + \Gamma)$ vanishes on every
curve contracted by $g_{i}$.

Furthermore, there will be only finitely many such $g_{i}: X_{i} \to Z_{i}$ realized by all the $\Gamma$
in $\mathcal{E}_{A}$.
\end{thrm}

Note that $V$ is not a subspace of $N^{1}(X)$; we need to have a finite dimensional
space of actual divisors, not numerical classes.
However, we would like to have a finiteness statement for certain
open subsets $U$ in $N^{1}(X)$.  So, we apply Lemma \ref{rescaleampleklt} to find
finitely many divisors such that
every element of $U$ has a klt representative in the space spanned by these divisors.
In this way we obtain:

\begin{prop} \label{mmpuse}
Let $(X,\Delta)$ be a $\mathbb{Q}$-factorial klt pair.
Suppose that $B$ is a big effective
$\mathbb{R}$-divisor such that $(X,\Delta + B)$ is klt.  Then there is some
neighborhood $U \subset N^{1}(X)$ of $[K_{X} + \Delta + B]$ and a finite set of movable
curves $\{ C_{i} \}$ so that every class $[\alpha] \in U$ that lies on the pseudo-effective boundary
satisfies $[\alpha] \cdot C_{i}=0$ for some $C_{i}$.
\end{prop}

\begin{proof}
First we apply Lemma \ref{rescaleampleklt} to restrict our attention to a finite set of
Weil divisors.

Since $B$ is big, it is numerically equivalent to $H + E$
for some ample $\mathbb{R}$-divisor $H$ and effective $\mathbb{R}$-divisor $E$.  For sufficiently small
$\tau > 0$, the pair $(X,\Delta + B + \tau E)$ is still klt, and then so is the pair
$(X,\Delta + (1-\tau)B + \tau E)$.  Let $A$ be some sufficiently small ample divisor so that
$\tau H - A$ is ample.  If we replace $A$ by some $\mathbb{R}$-linearly equivalent divisor,
we can ensure that
\begin{equation*}
\left( X,\Delta + A + (1-\tau) B + \tau E \right)
\end{equation*}
is klt.

Let $\{ H_{j} \}_{j=1}^{m}$ be a finite set of ample divisors such that the convex hull of the $[H_{j}]$
contains an open set around $[\tau H - A]$ in $N^{1}(X)$.  Choose an
open neighborhood $U'$ of $[B-A]$ that is contained in the convex hull of the classes
$[B - \tau H + H_{j}]$.
We may apply Lemma \ref{rescaleampleklt} to $(X,\Delta + A + (1-\tau)B + \tau E)$ and
the $H_{j}$ to obtain a finite set of ample divisors $W_{j}$.
Let $V$ be the vector space of $\mathbb{R}$-Weil divisors spanned by the irreducible components
of $(1-\tau)B + \tau E$ and of the $W_{j}$.  Thus, $V$ is a finite dimensional vector space
of Weil divisors so that every class in $U'$ has an effective representative
$\Gamma \in V$ with $(X,\Delta + A + \Gamma)$ klt.

Now we apply Theorem \ref{mmptheorem} with $V$ and $A$ as chosen.
According to the first part of the theorem, every class in $U'$ has a representative
$\Gamma$ so that, setting $D := K_{X} + \Delta + A + \Gamma$, we can run the $D$-minimal model program
with scaling.  If $D$ is on the pseudo-effective boundary, we obtain a birational
contraction $\phi_{i}: X \dashrightarrow X_{i}$ with a Mori fibration
$g_{i}: X_{i} \to Z_{i}$, where $K_{X_{i}} + \phi_{i*}(\Delta + A + \Gamma)$ vanishes on every
curve contracted by $g_{i}$.  Pick a curve $B_{i}$ in a general fiber of $g_{i}$ that
avoids the (codimension at least 2) locus where $\phi^{-1}: X_{i} \dashrightarrow X$ is not an isomorphism.
If we choose $B_{i}$ sufficiently general, then it belongs to a family of curves dominating $X_{i}$.
Define $C_{i}$ to be the image of $B_{i}$ under $\phi_{i}^{-1}$.  Of course $C_{i}$ is also a movable curve,
and since $\phi$ is an isomorphism on a neighborhood of $C_{i}$ we have $D \cdot C_{i} = 0$.

According to the second part of Theorem \ref{mmptheorem}, as we vary $\Gamma$
we obtain only finitely many Mori fibrations $g_{i}: X_{i} \to Z_{i}$.
Applying this construction to each fibration in turn yields a finite set of curves $\{ C_{i} \}$.
If we let $U$ be the open neighborhood of 
$K_{X} + \Delta + B$ given by $U' + [K_{X} + \Delta + A]$, then this finite
set of curves satisfies the conclusion.
\end{proof}

\begin{rmk} \label{rationalremark}
I am unable to show that the curves in Theorem \ref{firsttheorem} can be chosen to be rational curves.
In the proof of Proposition \ref{mmpuse}, we needed to choose curves
that avoided certain codimension $2$ subvarieties.  The obstacle to
finding rational curves with this property is that the ambient variety may be singular.
It is conjectured that one can find rational curves in the smooth locus of any $\mathbb{Q}$-Fano
variety with klt singularities, which would suffice for our purposes.

If we are willing to replace the $C_{i}$ by classes $\alpha_{i}$ that do not necessarily represent curves, we
can say something more.  Since a Mori fibration has relative Picard number $1$ and the general
fiber is a log Fano variety, the class of any curve on the fiber is proportional to the class
of a rational curve.  Thus we may choose $\alpha_{i}$ to be the numerical pullback of the class
of a nef rational curve.  See \cite{araujo08} for more details on numerical pullbacks.
\end{rmk}

Using compactness we can extend Proposition \ref{mmpuse} to a global result.

\begin{cor} \label{compactmmpuse}
Let $(X,\Delta)$ be a $\mathbb{Q}$-factorial klt pair.
Suppose that $\mathcal{S} \subset \overline{NE}^{1}(X)$ is a set of divisor classes satisfying
\begin{enumerate}
\item $\mathcal{S}$ is closed.
\item For each element $\beta \in \mathcal{S}$, there is some big effective divisor $B$ such that
$(X,\Delta + B)$ is klt and $[K_{X} + \Delta + B] = c \beta$ for some $c>0$.
\end{enumerate}
Then there is a finite set of movable curves $\{ C_{i} \}$ so
that every class $\alpha \in \mathcal{S}$
that lies on the pseudo-effective boundary satisfies $\alpha \cdot C_{i}=0$
for some $C_{i}$.
\end{cor}

\begin{proof}
Let $\mathcal{R}$ be the cone over $\mathcal{S}$.
Suppose that $\gamma \in \mathcal{R}$.
Since some positive multiple of $\gamma$ is in $\mathcal{S}$, we have
$[K_{X} + \Delta + B] = c \gamma$ for some $c>0$ and for some big effective divisor $B$
with $(X, \Delta + B)$ klt.  Apply Proposition \ref{mmpuse}
to $B$ to find an open neighborhood $U$ of $[K_{X} + \Delta + B]$ and finitely many movable curves
$C_{i}$ such that every $\alpha \in U$ on the pseudo-effective boundary
satisfies $\alpha \cdot C_{i} = 0$ for some $C_{i}$.
If we rescale $U$ by $1/c$ we find a
neighborhood $U_{\gamma}$ of $\gamma$ such that every $\alpha \in U_{\gamma}$ on the
pseudo-effective boundary satisfies $\alpha \cdot C_{i} = 0$ for some $C_{i}$.

Now fix some compact slice $Q$ of the cone $\overline{NE}^{1}(X)$.  Since
$\mathcal{S}$ is closed, $\mathcal{R} \cap Q$ is compact.
To each point $\gamma \in \mathcal{R} \cap Q$ we have associated finitely many curves $C_{i}$
and an open set $U_{\gamma}$ containing $\gamma$.  The $U_{\gamma}$ define an open cover of
$\mathcal{R} \cap Q$, so by compactness there is a finite subcover.  The finite set of corresponding curves
satisfies the conclusion.
\end{proof}

\section{The Cone of Nef Curves}

We now analyze the cone of nef curves.  Suppose that $(X,\Delta)$ is a klt pair.
We start by identifying certain rays in $\overline{NM}_{1}(X)$.

\begin{defn}
A \emph{coextremal} ray $\mathbb{R}_{\geq 0}[\alpha] \subset N_{1}(X)$
is a $(K_{X}+\Delta)$-negative ray of $\overline{NM}_{1}(X)$ that
is extremal for $\overline{NE}_{1}(X)_{K_{X} + \Delta \geq 0} + \overline{NM}_{1}(X)$.
That is, it must satisfy:
\begin{enumerate}
\item $\alpha \in \overline{NM}_{1}(X)$ has $(K_{X}+\Delta) \cdot \alpha < 0$.
\item If $\beta_{1},\beta_{2}$ are classes in
$\overline{NE}_{1}(X)_{K_{X} + \Delta \geq 0} + \overline{NM}_{1}(X)$ with
\begin{equation*}
\beta_{1} + \beta_{2} \in \mathbb{R}_{\geq 0}\alpha
\end{equation*}
then $\beta_{1},\beta_{2} \in \mathbb{R}_{\geq 0}\alpha$.
\end{enumerate}
\end{defn}

We also need a ``dual'' notion for divisors.

\begin{defn}
A \emph{bounding} divisor $D$ is any non-zero
$\mathbb{R}$-Cartier divisor $D$ satisfying the following properties:
\begin{enumerate}
\item $D \cdot \alpha \geq 0$ for every class $\alpha$ in $\overline{NE}_{1}(X)_{K_{X} + \Delta \geq 0} +
\overline{NM}_{1}(X)$.
\item $D^{\perp}$ contains some coextremal ray.
\end{enumerate}
The zero divisor is not considered to be a bounding divisor.

We will often need a more restrictive notion.  Suppose that
$V \subset N_{1}(X)$ is any subset.  If a bounding divisor
satisfies $D \cdot \alpha \geq 0$ for every $\alpha \in V$, I will call it
a $V$-bounding divisor.  Note that as $V$ gets larger, the set of $V$-bounding
divisors gets smaller.
\end{defn}

Bounding divisors support the cone
$\overline{NE}_{1}(X)_{K_{X} + \Delta \geq 0} + \overline{NM}_{1}(X)$ along
a face that includes some coextremal ray.  In particular, every bounding divisor is on the
pseudo-effective boundary.  The following lemma is essentially the same as \cite{araujo05}, Claim 4.2.

\begin{lem} \label{boundingform}
Let $(X,\Delta)$ be a klt pair such that $K_{X} + \Delta$ is not pseudo-effective.
Let $V$ be a closed convex cone containing $\overline{NE}_{1}(X)_{K_{X} + \Delta = 0} - \{ 0 \}$ in
its interior.  Then every $V$-bounding divisor $D$ can be written as
\begin{equation*}
D = \delta_{D} (K_{X} + \Delta) + A_{D}
\end{equation*}
for some ample $A_{D}$ and some $\delta_{D} > 0$.
\end{lem}

For convenience we will recall the proof.

\begin{proof}
Following \cite{araujo05}, we suppose the lemma fails and derive a contradiction.
That is, suppose there is some $V$-bounding divisor $D$ such that the interior of the cone
\begin{equation*}
\sigma = \mathbb{R}_{\geq 0} [D] + \mathbb{R}_{\geq 0} [-K_{X} - \Delta]
\end{equation*}
never intersects the ample cone.  Then there is a curve class $\alpha$ for which
the cone $\sigma$ is contained in $\alpha_{\leq 0}$, but the ample cone is contained in $\alpha_{> 0}$.
By Kleiman's criterion $\alpha$ is in the closed cone of effective curves;
in particular $\alpha \in \overline{NE}_{1}(X)_{K_{X} + \Delta \geq 0}$.

Because $D$ is a $V$-bounding divisor and $V$ is not contained in any hyperplane, $D$ must be positive
on the interior of $V$.  So $D$ is positive on $\overline{NE}_{1}(X)_{K_{X} + \Delta = 0} - \{ 0 \}$, and thus
also positive on all of $\overline{NE}_{1}(X)_{K_{X} + \Delta \geq 0} - \{ 0 \}$.  In particular we must
have $D \cdot \alpha > 0$, contradicting $\sigma \subset \alpha_{\leq 0}$.
\end{proof}

Now we apply the results of the previous section.

\begin{prop} \label{Vfiniteness}
Let $(X,\Delta)$ be a $\mathbb{Q}$-factorial klt pair.
Let $V$ be a closed convex cone containing $\overline{NE}_{1}(X)_{K_{X} + \Delta = 0} - \{ 0 \}$ in its
interior.  There is a finite set of movable curves $\{ C_{i} \}$ such that for any $V$-bounding divisor $D$
there is some $C_{i}$ for which $D \cdot C_{i} = 0$.
\end{prop}

\begin{proof}
The statement is vacuous when the set of $V$-bounding divisors is empty (for example,
when $K_{X}+\Delta$ is pseudo-effective or when $V$ is the entire space), so we assume otherwise.

We apply Corollary \ref{compactmmpuse} to the set of $V$-bounding divisors.
The first hypothesis holds since $V$-bounding divisors are precisely
the divisors $D$ on the pseudo-effective boundary satisfying the closed condition
\begin{equation*}
D \cdot \alpha \geq 0 \textrm{ for every class } \alpha \in
\overline{NE}_{1}(X)_{K_{X} + \Delta \geq 0} + V + \overline{NM}_{1}(X).
\end{equation*}
We verify the second hypothesis by using
Lemma \ref{boundingform}.  Each $V$-bounding divisor $D$ satisfies
\begin{equation*}
\frac{1}{\delta_{D}}D = K_{X} + \Delta + \frac{1}{\delta_{D}}A_{D}
\end{equation*}
for some ample $A_{D}$ and some $\delta_{D} > 0$.
By replacing $\frac{1}{\delta_{D}}A$ by some $\mathbb{R}$-linearly equivalent effective divisor $A'$,
we may ensure that $(X,\Delta + A')$ is a klt pair.  Thus the second hypothesis holds as well.
Since every $V$-bounding divisor is on the pseudo-effective boundary, an application
of Corollary \ref{compactmmpuse} finishes the proof.
\end{proof}

\begin{cor} \label{qfactorialnefcone}
Let $(X,\Delta)$ be a $\mathbb{Q}$-factorial klt pair, and let
$V$ be a closed convex cone containing $\overline{NE}_{1}(X)_{K_{X} + \Delta = 0} - \{ 0 \}$ in its
interior.  There are finitely many movable curves $C_{i}$ such that
\begin{equation*}
\overline{NE}_{1}(X)_{K_{X} + \Delta \geq 0} + V + \overline{NM}_{1}(X) =
\overline{NE}_{1}(X)_{K_{X} + \Delta \geq 0} + V + \sum_{i=1}^{N} \mathbb{R}_{\geq 0} [C_{i}].
\end{equation*}
\end{cor}

\begin{proof}
Note that as $V$ increases, the statement of the theorem becomes easier to prove.  Therefore,
we may assume by shrinking $V$ that $\overline{NE}_{1}(X) + V$ does not contain any $1$-dimensional
subspace of $N_{1}(X)$.  In particular, this implies that there is some ample divisor $A$ that is
positive on $V - \{ 0 \}$.

Let $\{ C_{i} \}$ be the set of curves found in Proposition \ref{Vfiniteness}.
Suppose that there were some curve class $\alpha \in \overline{NM}_{1}(X)$ not contained in
$\overline{NE}_{1}(X)_{K_{X} + \Delta \geq 0} + V + \sum \mathbb{R}_{\geq 0} [C_{i}]$.
In particular, there is some divisor $B$ that is positive on
$\overline{NE}_{1}(X)_{K_{X} + \Delta \geq 0} + V + \sum \mathbb{R}_{\geq 0} [C_{i}]$,
but for which $B \cdot \alpha < 0$.

Let $A$ be an ample divisor positive on $V - \{ 0 \}$, and consider $A + \tau B$,
where $\tau > 0$ is the maximum over all $t$ such that $A + tB$ is pseudoeffective.  Then $A + \tau B$ is a
$V$-bounding divisor, but $(A + \tau B) \cdot C_{i} > 0$ for every $C_{i}$, a contradiction.
\end{proof}

We now have a finiteness statement for coextremal rays associated to $\mathbb{Q}$-factorial klt pairs $(X,\Delta)$.
This finiteness statement is essentially equivalent to Theorem \ref{firsttheorem}.
Thus, we only need to reduce the problem for general dlt pairs $(X,\Delta)$
to this specific case.  The following lemma from \cite{km98} shows that dlt pairs can be approximated by
klt pairs on any projective variety.  (Although \cite{km98} assumes
that $K_{X} + \Delta$ is $\mathbb{Q}$-Cartier, the proof works
equally well for $\mathbb{R}$-Cartier divisors.)

\begin{lem}[\cite{km98}, Proposition 2.43] \label{dltpairs}
Suppose that $(X,\Delta)$ is a dlt pair and that $A$ is an ample Cartier divisor on $X$.
Let $\Delta_{1}$ be an effective $\mathbb{Q}$-Weil divisor such that
$\Supp(\Delta) = \Supp(\Delta_{1})$.
Then there is some rational $c>0$ and effective $\mathbb{Q}$-divisor $D$,
numerically equivalent to $\Delta_{1} + cA$,
such that $(X,\Delta + \tau D - \tau \Delta_{1})$ is klt for any sufficiently small $\tau > 0$.
\end{lem}

We can now extend Corollary \ref{qfactorialnefcone} to the dlt case and remove the assumption of $\mathbb{Q}$-factoriality.

\begin{thrm} \label{nefcone}
Let $(X,\Delta)$ be a dlt pair and let
$V$ be a closed convex cone containing $\overline{NE}_{1}(X)_{K_{X} + \Delta = 0} - \{ 0 \}$ in its interior.
There are finitely many movable curves $C_{i}$ such that
\begin{equation*}
\overline{NE}_{1}(X)_{K_{X} + \Delta \geq 0} + V + \overline{NM}_{1}(X) =
\overline{NE}_{1}(X)_{K_{X} + \Delta \geq 0} + V + \sum_{i=1}^{N} \mathbb{R}_{\geq 0} [C_{i}].
\end{equation*}
\end{thrm}

\begin{proof}
We assume that $K_{X} + \Delta$ is not pseudo-effective, as otherwise
the statement is vacuous.

Let $A$ be some ample Cartier divisor on $X$ and $\Delta_{1}$ some effective $\mathbb{Q}$-Weil divisor
with the same support as $\Delta$.  Choose $c$ and $D$ as in Lemma \ref{dltpairs}.
Suppose that we choose $\tau$ small enough so that
\begin{equation*}
\overline{NE}_{1}(X)_{K_{X} + \Delta + \tau D - \tau \Delta_{1} = 0} \subset V
\end{equation*}
Then the result for $(\Delta,V)$ follows from the result
for $(\Delta + \tau D - \tau \Delta_{1},V)$.  Thus we may assume that $(X,\Delta)$ is klt.

Suppose that $(X,\Delta)$ is not $\mathbb{Q}$-factorial.  By Corollary 1.4.3 of
\cite{bchm06} we may find a log terminal model $(Y,\Gamma)$ over $(X,\Delta)$.  That is,
$(Y,\Gamma)$ is a $\mathbb{Q}$-factorial klt pair, and there is a small birational map
$\pi: Y \to X$ such that $\pi_{*}\Gamma = \Delta$ and $K_{Y} + \Gamma = \pi^{*}(K_{X} + \Delta)$.
Since the map $\pi_{*}: N_{1}(Y) \to N_{1}(X)$
is linear, $\pi_{*}^{-1}V$ is still closed and convex.  Furthermore, $\pi_{*}^{-1}V$ contains
$\overline{NE}_{1}(Y)_{K_{Y} + \Gamma = 0}$ in its interior, since $\pi_{*}$ on curves
is dual to $\pi^{*}$ on Cartier divisors.

Apply Corollary \ref{qfactorialnefcone} to $(Y,\Gamma)$ and $\pi_{*}^{-1}V$
to obtain the equality of cones
\begin{equation*}
\overline{NE}_{1}(Y)_{K_{Y} + \Gamma \geq 0} + \pi_{*}^{-1}V + \overline{NM}_{1}(Y) =
\overline{NE}_{1}(Y)_{K_{Y} + \Gamma \geq 0} + \pi_{*}^{-1}V + \sum_{i=1}^{N} \mathbb{R}_{\geq 0} [C_{i}].
\end{equation*}
If we take the image of these cones under the map $\pi_{*}$ and apply Lemma \ref{conepushforward}, we
obtain
\begin{equation*}
\overline{NE}_{1}(X)_{K_{X} + \Delta \geq 0} + V + \overline{NM}_{1}(X) =
\overline{NE}_{1}(X)_{K_{X} + \Delta \geq 0} + V + \sum_{i=1}^{N} \mathbb{R}_{\geq 0} [\pi_{*}C_{i}].
\end{equation*}
The pushforward of a movable curve is again movable.
Thus the coextremal rays are still spanned by curves belonging to covering families, giving
the result of Theorem \ref{nefcone} for $(X,\Delta)$.
\end{proof}

\begin{proof}[Proof of Theorem \ref{firsttheorem}:]

Let $\{ V_{j} \}$ be a countable set of nested closed convex cones containing
$\overline{NE}_{1}(X)_{K_{X} + \Delta = 0} - \{ 0 \}$ in their interiors such that
\begin{equation*}
\bigcap_{j} V_{j} = \overline{NE}_{1}(X)_{K_{X} + \Delta = 0}.
\end{equation*}
Let $\mathcal{A}_{j}$ be the finite set of curves
found by applying Theorem \ref{nefcone} to $(X,\Delta)$ and $V_{j}$.
By tossing out redundant curves, we may ensure that each curve in $\mathcal{A}_{j}$
generates a coextremal ray.
We define the countable set of curves $\mathcal{A} = \cup_{j} \mathcal{A}_{j}$.

We first show the equality of cones.
Suppose that there is some curve class $\alpha \in \overline{NM}_{1}(X)$
such that
\begin{equation*}
\alpha \notin \overline{NE}_{1}(X)_{K_{X} + \Delta \geq 0} + \overline{\sum_{\mathcal{A}}
\mathbb{R}_{\geq 0} [C_{i}]}.
\end{equation*}
Since this cone is closed and convex, there is a convex open neighborhood $U$ of the cone which also does not contain
$\alpha$.  For sufficiently high $j$, we have $V_{j} \subset U$, so
\begin{equation*}
\alpha \notin \overline{NE}_{1}(X)_{K_{X} + \Delta \geq 0} + V_{j} + \overline{\sum_{\mathcal{A}}
\mathbb{R}_{\geq 0} [C_{i}]}.
\end{equation*}
But this contradicts Theorem \ref{nefcone}.  This proves the non-trivial containment of Theorem \ref{firsttheorem}.

We also need to verify the accumulation condition for the rays generated by curves in $\mathcal{A}$.
Suppose that $\alpha$ is a point on the $(K_{X} + \Delta)$-negative
portion of the boundary of $\overline{NE}_{1}(X)_{K_{X} + \Delta \geq 0} + \overline{NM}_{1}(X)$
and that $\alpha$ does not lie on a hyperplane supporting both $\overline{NM}_{1}(X)$ and
$\overline{NE}_{1}(X)_{K_{X} + \Delta \geq 0}$.  For a sufficiently small open neighborhood $U$ of
$\alpha$ the points of $\overline{U}$ still do not lie on such a hyperplane.  We may also assume that
$\overline{U}$ is disjoint from $\overline{NE}_{1}(X)_{K_{X} + \Delta \geq 0}$.  We define
\begin{equation*}
\mathcal{P} := \overline{U} \cap \partial \left( \overline{NE}_{1}(X)_{K_{X} + \Delta \geq 0} + \overline{NM}_{1}(X) \right).
\end{equation*}

Fix a compact slice of $\overline{NE}^{1}(X)$ and let $\mathcal{D}$ denote the bounding divisors
on this slice that have vanishing intersection with some element of $\mathcal{P}$.
By construction $\mathcal{D}$ is positive on $\overline{NE}_{1}(X)_{K_{X} + \Delta \geq 0} - \{ 0 \}$.
By passing to a compact slice it is easy to see that $\mathcal{D}$ is also positive on $V_{j} - \{ 0 \}$ for
a sufficiently large $j$.  In other words, every element of $\mathcal{P}$ is on the boundary of
$\overline{NE}_{1}(X)_{K_{X} + \Delta \geq 0} + V_{j} + \overline{NM}_{1}(X)$.  By
Theorem \ref{nefcone} there are only finitely many coextremal rays that lie on this cone, and thus
only finitely many coextremal rays through $U$.  So $\alpha$
can not be an accumulation point.
\end{proof}

\subsection*{Cutkosky's Example} \label{examplesection}

In this section we give an example of a threefold for which the stronger statement
\begin{equation*} \label{movablecone} \tag{*}
\overline{NM}_{1}(X)=
\overline{NM}_{1}(X)_{K_{X} + \Delta \geq 0} + \sum \mathbb{R}_{\geq 0} [C_{i}]
\end{equation*}
does not hold for any locally discrete countable collection of curves.  Although this statement seems to be a closer
analogue of the Cone Theorem, it is much less natural from the viewpoint of the minimal model program.  The problem is that the bounding divisors are no longer of the form $K_{X}+ \Delta + A$ for an ample $A$, but
$K_{X}+ \Delta + B$ for a big $B$.  In general we cannot say anything about the singularities of such divisors
and so we can not apply Proposition \ref{mmpuse}.

\begin{rmk}
By dualizing we can transform \eqref{movablecone} to an equivalent statement concerning the structure
of the pseudo-effective cone of divisors.  To be more precise, \eqref{movablecone} implies that
the boundary of $\overline{NE}^{1}(X)$ is polyhedral inside any proper open subcone of the cone spanned by $\overline{NE}^{1}(X)$ and $[K_{X}+\Delta]$.  Note that Theorem \ref{nefcone} guarantees this polyhedral structure
along portions of the boundary of the nef cone of divisors.

It is possible that \eqref{movablecone} holds for surfaces.  It is shown in \cite{bks04}
that for a surface the portion of the pseudo-effective boundary that is not nef is locally polyhedral.
Thus, one would need to verify that coextremal rays do not accumulate where
the pseudo-effective cone and nef cone first coincide.
More generally, \cite{nakayama04} and \cite{boucksom04} use the divisorial Zariski decomposition to
prove that some portions of the boundary of $\overline{NE}^{1}(X)$ are polyhedral.
\end{rmk}

\begin{exmple} \label{cutkoskyexample}
We construct a threefold for which \eqref{movablecone} does not hold.
In particular, we construct a smooth variety $X$ such that
$-K_{X}$ is big but $\overline{NM}_{1}(X)$ is circular along certain portions of its boundary.
This example is due to Cutkosky, who uses it to find a divisor with no rational Zariski decomposition
(see \cite{cutkosky86}).

Let $Y$ be an abelian surface with Picard number at least $3$.
As $Y$ is abelian, the nef cone of divisors and the pseudo-effective cone of divisors coincide.  This
cone is circular in $N^{1}(Y)$: it consists of all the curve classes with non-negative
self intersection and non-negative intersection with some ample divisor.

Choose a divisor $L$ on $Y$ such that $-L$ is ample and
define $X$ to be the $\mathbb{P}^{1}$-bundle $X := \mathbb{P}_{Y}(\mathcal{O} \oplus \mathcal{O}(L))$
with projection $\pi: X \to Y$.  Let $S$ denote the zero section of $\pi$, that is, the
section such that $S|_{S}$ is the linear equivalence class of $L$.  Every divisor on $X$ can be written as
$aS + \pi^{*}D$ for some integer $a$ and some divisor $D$ on $Y$.  Using adjunction
we find $K_{X} = -2S + \pi^{*}L$.

By the general theory of $\mathbb{P}^{1}$-bundles, a divisor of the form $S + \pi^{*}D$ is pseudo-effective on $X$ iff there is a pseudo-effective divisor on $Y$ in the cone generated by $D$ and $D+L$.
Since $-L$ is ample, this amounts to requiring that $D$ be pseudo-effective.
Thus the pseudo-effective cone $\overline{NE}^{1}(X)$ is generated by $S$ and
$\pi^{*} \overline{NE}^{1}(Y)$.  Note that $-K_{X}$ is big since it lies in the interior of this cone.

Since $\overline{NE}^{1}(Y)$ is circular, the corresponding portion of the boundary of $\overline{NE}^{1}(X)$ is not the closure of any locally discrete collection of rays.  By duality $\overline{NM}_{1}(X)$ is also not the closure of any locally discrete collection of rays. Since $-K_{X}$ is big, all of $\overline{NM}_{1}(X)$ is $K_{X}$-negative, so \eqref{movablecone} fails for $X$.
$\Box$
\end{exmple}

\section{$K_{X}$-negative Faces and the Minimal Model Program}

In this section we prove Theorem \ref{contraction} which relates $K_{X}$-negative extremal faces of the cone
$\overline{NE}_{1}(X)_{K_{X} + \Delta \geq 0} + \overline{NM}_{1}(X)$ to outcomes of the minimal model
program.  The first step is to prove a version of Theorem \ref{contraction} for divisors
$K_{X} + \Delta$ with $\Delta$ big.

\begin{thrm} \label{bigcontraction}
Let $(X,\Delta)$ be a $\mathbb{Q}$-factorial klt pair with $\Delta$ big.  Suppose that $D := K_{X} + \Delta$ lies
on the pseudo-effective boundary.  Then there is a birational morphism $\psi: W \to X$
and a contraction $h: W \to Z$ such that:
\begin{enumerate}
\item Every movable curve $C$ on $W$ with $\psi^{*}D \cdot C = 0$ is contracted by $h$.
\item For a general pair of points in a general fiber of $h$, there is a movable
curve $C$ through the two points with $\psi^{*}D \cdot C = 0$.
\end{enumerate}
These properties determine the pair $(W,h)$ up to birational equivalence.
In fact the map we construct satisfies a stronger property:
\begin{itemize}
\item[(3)] There is an open set $U \subset W$ such that a complete curve $C$ in $U$ is contracted by $h$ iff $\psi^{*}D \cdot C = 0$ and $[C] \in \overline{NM}_{1}(W)$.  For a general fiber $H$ of $h$ there is a completion $\overline{H}$ of $U \cap H$ so that the complement of $U \cap H$ has codimension at least $2$.  Furthermore $\psi$ is an isomorphism on $U$.
\end{itemize}
\end{thrm}

The proof goes as follows.  By running the $(K_{X} + \Delta)$-minimal model program
we obtain a birational contraction $\phi: X \dashrightarrow X'$ such that
$K_{X'} + \phi_{*}\Delta$ is nef.  Applying the basepoint free theorem to $X'$
gives us the desired map on an open subset of $X$.

There is one subtlety: in order to prove the stronger Property (3), we will need to
improve the properties of the contraction $g': X' \to Z$.
This is accomplished by the following lemma.

\begin{lem}  \label{relativelogfano}
Let $(X',\Delta')$ be a $\mathbb{Q}$-factorial klt pair such that $\Delta'$ is big.
Suppose that $g': X' \to Z$ is a $(K_{X'} + \Delta')$-trivial fibration.  We can construct a
birational contraction $\phi': X' \dashrightarrow X^{+}$ over $Z$ (with morphism $g^{+}: X^{+} \to Z$)
and a Zariski-closed subset $N^{+} \subset X^{+}$ of codimension at least $2$ such that every curve
contracted by $g^{+}$ that does not lie in $N^{+}$ is nef.
\end{lem}

\begin{proof}
Since $\Delta'$ is big and $g'$ is $(K_{X'} + \Delta')$-trivial, $g'$ is a log Fano fibration.
Thus, $X'$ admits only finitely many birational contractions over $Z$ (see \cite{bchm06}, Corollary 1.3.1).
If we let $N'$ be the union of the (finitely many) closed subsets where these birational maps
are not isomorphisms, then any curve $C$ contracted by $g'$ and avoiding $N'$ is
nef.  We'll call $N'$ the ``bad'' locus of $g': X' \to Z$.

Let $\{ E_{i} \}$ denote the divisorial components of $N'$.  Through any point $x \in E_{i}$ there
is some curve $C$ that is not nef.  If we pick
$x$ to be very general in $E_{i}$, then $C$ deforms to cover $E_{i}$.  The only
way that $C$ could not be nef is if $E_{i} \cdot C < 0$.  Thus, every component $E_{i}$
is covered by curves with $E_{i} \cdot C < 0$.

Fix some component $E$, and choose $\epsilon > 0$ such that $(X',\Delta' + \epsilon E)$ is still klt.
Note that $K_{X'} + \Delta'+ \epsilon E$ is numerically equivalent over $Z$ to $\epsilon E$.
Thus when we run the relative $(K_{X'} + \Delta' + \epsilon E)$-minimal model program,
we must contract the component $E$.  By performing this process inductively, we eventually
find a variety $X^{+}$ such that the ``bad'' locus $N^{+}$ has codimension at least $2$.
\end{proof}

\begin{proof}[Proof of Theorem \ref{bigcontraction}:]
Since $\Delta$ is big, we can run the $(K_{X} + \Delta)$-minimal
model program with scaling to obtain a birational map $\phi: X \dashrightarrow X'$ such
that $K_{X'} + \phi_{*}\Delta$ is nef.  By applying the basepoint freeness theorem to
$K_{X'} + \phi_{*}\Delta$ we obtain a contraction morphism $g': X' \to Z$.
Applying Lemma \ref{relativelogfano} we obtain a birational contraction
$\phi': X' \dashrightarrow X^{+}$ and $g^{+}: X^{+} \to Z$.  There is a codimension 2 subset
$N^{+} \subset X^{+}$ such that every curve that is contracted by $g^{+}$ and not contained in $N^{+}$ is nef.
We denote the composition $\phi' \circ \phi: X \dashrightarrow X^{+}$ by $\Phi$.  We define
$W$ by taking a resolution of the map $\Phi$, so that we have maps $s: W \to X$ and $s^{+}: W \to X^{+}$.
We let $h = g^{+} \circ s^{+}: W \to Z$.

Note that the map $\Phi$ is $(K_{X} + \Delta)$-non-positive, since $\phi$ is $(K_{X} + \Delta)$-negative
and $\phi'$ is $(K_{X} + \Delta)$-trivial.  Thus, for some effective divisor $E$ we have
\begin{equation*}
s^{*}(K_{X} + \Delta) = s^{+*}(K_{X^{+}} + \Phi_{*}\Delta) + E
\end{equation*}
(see \cite{km98}, Lemma 3.38).
As a consequence, if $C$ is a movable curve on $W$ with $s^{*}D \cdot C = 0$, then we also have
$s^{+*}(K_{X^{+}} + \Phi_{*}\Delta) \cdot C = 0$.  In particular, $C$ is contracted by $h$, showing
Property (1).

Let $U_{X}$ be the open subset of $X$ on which $\Phi$ is an isomorphism.  Since
$\Phi^{-1}$ does not contract a divisor, the complement of
$\Phi(U_{X})$ has codimension at least $2$.
By shrinking $U_{X}$ we can remove all fibers of $g^{+}$ on which the complement of $\Phi(U_{X})$
has codimension $1$ and all the reducible fibers of $g^{+}$.
Finally, we remove the subset $\Phi^{-1}(N^{+} \cap \Phi(U_{X}))$; note that $\Phi(U_{X})$ still has
codimension $2$ in a general fiber of $g^{+}$.  We define the open subset $U \subset W$ by taking $s^{-1}(U_{X})$.

Since $s^{+}(U)$ has codimension $2$ in a general fiber of $g^{+}$, we can connect two general points
in a general fiber of $g^{+}$ by a movable curve $C^{+} \subset s^{+}(U)$: we just intersect the
fiber with general very ample divisors that contain the two points.  Since
$\Phi$ is an isomorphism on $U$, the movable curve $\Phi^{-1}C^{+}$ satisfies
$\psi^{*}D \cdot \Phi^{-1}C^{+} = (K_{X+} + \Phi_{*}\Delta) \cdot C^{+} = 0$. This shows Property (2).

In fact, suppose that $C \subset U$ is a complete curve contracted by $h$.  By construction $s^{+}(C)$
is nef, so by Lemma \ref{conepushforward} $C$ is also nef.
Furthermore, since $s$ and $s^{+}$ are isomorphisms on a neighborhood of $C$, we must have
$\psi^{*}D \cdot C = s^{+*}(K_{X^{+}} + \Phi_{*}\Delta) \cdot C = 0$.
Conversely, any curve $C$ on $W$ with $s^{+*}(K_{X+} + \Phi_{*}\Delta) \cdot C = 0$
must be contracted by $h$, showing the first condition of Property (3).  To see the second, suppose that $H$ is a general fiber of $h$ and let $\overline{H}$ denote the closure of $s^{+}(U \cap H)$ in $X^{+}$.  The complement of $U \cap H$ in $\overline{H}$ has codimension at least $2$.

Finally, we must show the uniqueness of $h: W \to Z$ up to birational equivalence.  So, suppose that
$h': W' \to Z'$ is another map satisfying Properties (1) and (2).  Let $\widetilde{W}$ be a common resolution.
The maps $\widetilde{W} \to Z$ and $\widetilde{W} \to Z'$ still satisfy Properties (1) and (2).  Since a contraction morphism is determined up to birational equivalence by the movable curves it contracts, $h$ and $h'$
coincide on an open subset.
\end{proof}

When we rephrase Theorem \ref{bigcontraction} in terms of faces of
$\overline{NE}_{1}(X)_{K_{X} + \Delta \geq 0} + \overline{NM}_{1}(X)$, we obtain
Theorem \ref{contraction}.  The main difficulty is to show that most $(K_{X} + \Delta)$-negative
extremal faces of the cone admit a divisor $D$ supporting the cone precisely along that face.

\begin{proof}[Proof of Theorem \ref{contraction}:]
We first reduce to the case when $(X,\Delta)$ is a $\mathbb{Q}$-factorial klt pair.
Fix divisors $A$ and $\Delta_{1}$ on $X$ and choose $c$ and $D$ as in Lemma \ref{dltpairs}.  For appropriate choices we may ensure that $F$ is  $(\Delta + \tau D - \tau \Delta_{1})$-negative and that $\beta^{\perp}$ avoids the cone $\overline{NE}_{1}(X)_{K_{X} + \Delta + \tau D - \tau \Delta_{1} \geq 0}$.  Thus we may assume that $(X,\Delta)$ is klt.  Let $\pi: (Y,\Gamma) \to (X,\Delta)$ be a log terminal model.  By intersecting $\overline{NM}_{1}(Y)$ with $\pi_{*}^{-1}(F)$ we obtain a face $F_{Y}$.  Note that $\pi^{*}\beta^{\perp}$ avoids $\overline{NE}_{1}(Y)_{K_{Y} + \Gamma \geq 0}$ and contains $F_{Y}$.  Suppose we knew the statement of Theorem \ref{contraction} for $(Y,\Gamma)$, $F_{Y}$, and $\pi^{*}\beta$.  Properties (1) and (2) for $(X,\Delta)$ follow immediately.  Since $\pi$ is a small contraction and $\psi: W \to Y$ is an isomorphism on $U$, we can remove the preimage of the $\pi$-exceptional locus from $U$ without affecting Property (3).  Thus we may assume that $(X,\Delta)$ is $\mathbb{Q}$-factorial.

Our next goal is to find a pseudo-effective divisor $D$ such that $D^{\perp}$ supports
$\overline{NE}_{1}(X)_{K_{X} + \Delta \geq 0} + \overline{NM}_{1}(X)$ exactly along $F$.
This will follow from our assumption that there is some pseudo-effective divisor class $\beta$ such that
$\beta^{\perp}$ contains $F$ but doesn't intersect $\overline{NE}_{1}(X)_{K_{X} + \Delta \geq 0}$.

We let $F_{\beta}$ denote the face
\begin{equation*}
F_{\beta} := \beta^{\perp} \cap \left( \overline{NE}_{1}(X)_{K_{X} + \Delta \geq 0} +
\overline{NM}_{1}(X) \right).
\end{equation*}
Note that $F$ is a subface of $F_{\beta}$.  For a sufficiently small closed convex cone $V$ containing
$\overline{NE}_{1}(X)_{K_{X} + \Delta = 0} - \{ 0 \}$ in its interior, $\beta$ is still positive on
\begin{equation*}
\left(\overline{NE}_{1}(X)_{K_{X} + \Delta \geq 0} + V \right) - \{ 0 \}.
\end{equation*}
This means that $F_{\beta}$ is also a $(K_{X} + \Delta)$-negative face of the larger cone
$\overline{NE}_{1}(X)_{K_{X} + \Delta \geq 0} + V + \overline{NM}_{1}(X)$.
Theorem \ref{nefcone} implies that there are only finitely many coextremal rays in a neighborhood
of $F_{\beta}$.  Thus for any subface of $F_{\beta}$ there is an $\mathbb{R}$-Cartier divisor $D$
such that $D^{\perp}$ supports $\overline{NE}_{1}(X)_{K_{X} + \Delta \geq 0} + \overline{NM}_{1}(X)$
exactly along that subface.  In particular this is true for the subface $F$.  Furthermore, $D$
must be pseudo-effective since it is non-negative on $\overline{NM}_{1}(X)$.  This finishes the construction
of $D$.

Since $D$ is positive on $\overline{NE}_{1}(X)_{K_{X} + \Delta = 0}$,
Lemma \ref{boundingform} shows that (after rescaling) $D = K_{X} + \Delta + A$ for some ample $A$.
Thus, we can apply Theorem \ref{bigcontraction} to $D$ to obtain $\psi: W \to X$ and $h: W \to Z$.  Properties
(1) and (2) and the uniqueness up to birational equivalence follow immediately.  To show Property (3),
we just need to note that since $\psi$ is an isomorphism on $U$,
for any irreducible curve $C$ on $U$ we have
$[\psi_{*}C] \in F$ iff $\psi^{*}D \cdot C = 0$ and $[C] \in \overline{NM}_{1}(W)$.
\end{proof}

\section{Accumulation of Rays} \label{accumulationsection}

In this section we consider whether coextremal rays accumulate only along
$\overline{NE}_{1}(X)_{K_{X} + \Delta = 0}$ as in Conjecture \ref{batyrevconjecture}.
This stronger statement is proven for terminal threefolds in \cite{araujo08}.  Araujo first finds
movable curves generating coextremal rays
by running the minimal model program.  Using boundedness of terminal threefolds of Picard number $1$, she
bounds the degrees of all of these curves with respect to a fixed polarization.  So in fact, for any
ample divisor $A$ there are only finitely many $(K_{X} + \Delta + A)$-negative coextremal rays.

The situation in higher dimensions is expected to be similar.  The following conjecture
is due to Alexeev (\cite{alexeev94}, 11.3), and A.~Borisov and L.~Borisov (\cite{borisov96}, Conjecture 1.1).

\begin{conj}[Borisov-Alexeev-Borisov]
For any $\epsilon > 0$, the family of $\mathbb{Q}$-Fano varieties of a given dimension
with log discrepancy greater than $\epsilon$ is bounded.
\end{conj}

Since running the minimal model program can only increase log discrepancies, this
conjecture combined with Araujo's argument should settle the question.

We consider Conjecture \ref{batyrevconjecture} from a different perspective.  Our goal
is to show that when $\Delta$ is big, Conjecture \ref{batyrevconjecture} can be derived from the termination of flips conjecture.

\begin{conj}[Termination of Flips] \label{terminationofflips}
Let $(X,\Delta)$ be a $\mathbb{Q}$-factorial klt pair.  Then there is no infinite sequence
of $(K_{X} + \Delta)$ flips.
\end{conj}

\begin{thrm} \label{bigcase}
Let $(X,\Delta)$ be a $\mathbb{Q}$-factorial klt pair with $\Delta$ big.  Assume that Conjecture \ref{terminationofflips} holds.  Then there are finitely many $(K_{X}+\Delta)$-negative movable curves $C_{i}$ such that
\begin{equation*}
\overline{NE}_{1}(X)_{K_{X} + \Delta \geq 0} + \overline{NM}_{1}(X) =
\overline{NE}_{1}(X)_{K_{X} + \Delta \geq 0} + \sum \mathbb{R}_{\geq 0} [C_{i}].
\end{equation*}
\end{thrm}

\begin{rmk}
Note that (in contrast to the Cone Theorem) Conjecture \ref{batyrevconjecture} does not directly follow from the special case when $\Delta$ is big.  The problem is the presence of the term
$\overline{NE}_{1}(X)_{K_{X} + \Delta \geq 0}$.  If we add a small ample divisor $\epsilon H$
to $\Delta$, we are actually changing the shape of this cone, so that the limiting behavior
as $\epsilon$ vanishes is more subtle.
\end{rmk}

The proof of Theorem \ref{bigcase} uses the same arguments as before.  Given a
bounding divisor $D$, we find a birational contraction $\phi: X \dashrightarrow X'$ and a Mori fibration
$g: X' \to Z$ that is $\phi_{*}D$-trivial.  We then use termination of flips to prove that
we need only finitely many fibrations.

Our first lemma is an analogue of Lemma \ref{boundingform} for an arbitrary bounding divisor.

\begin{lem} \label{boundingform2}
Let $(X,\Delta)$ be a klt pair such that $K_{X} + \Delta$ is not pseudo-effective.
Every bounding divisor $D$ can be written
\begin{equation*}
D = \delta_{D} (K_{X} + \Delta) + N_{D}
\end{equation*}
for some nef $N_{D}$ and some $\delta_{D} \geq 0$.
\end{lem}

\begin{proof}
Just as before, we suppose the lemma fails and derive a contradiction.
That is, suppose there is some bounding divisor $D$ such that the cone
\begin{equation*}
\sigma = \mathbb{R}_{\geq 0} [D] + \mathbb{R}_{\geq 0} [-K_{X} - \Delta]
\end{equation*}
never intersects the nef cone.  Then there is a curve class $\alpha$ for which
the cone $\sigma$ is contained in $\alpha_{< 0}$, but the nef cone is contained in $\alpha_{> 0}$.
By Kleiman's criterion $\alpha$ is in the closed cone of effective curves;
in particular $\alpha \in \overline{NE}_{1}(X)_{K_{X} + \Delta \geq 0}$.
Because $D$ is a bounding divisor, we should have $D \cdot \alpha \geq 0$.  This contradicts
$\sigma \subset \alpha_{< 0}$.
\end{proof}

Given a bounding divisor $D$, we want to find curves that have vanishing intersection with $D$ by
running the minimal model program.  During the process, we must ensure that $D$ is still a bounding
divisor, or at least that it still intersects some coextremal ray.  This issue is handled by the next lemma.

\begin{lem} \label{vanishingintersectionpushforward}
Let $(X,\Delta)$ be a $\mathbb{Q}$-factorial klt pair.  Suppose $D$ is a bounding divisor
and $\phi: X \dashrightarrow X'$ is a composition of $(K_{X} + \Delta)$ flips
and divisorial contractions that are $D$-non-positive.  Then there is a
class $\alpha$ on the boundary of $\overline{NM}_{1}(X')$ such that $\phi_{*}D \cdot \alpha = 0$
and $(K_{X'} + \phi_{*}\Delta) \cdot \alpha < 0$.
\end{lem}

\begin{proof}
Let $Y$ be a smooth variety resolving the rational map $\phi: X \dashrightarrow X'$.  We denote
the corresponding maps by $g: Y \to X$ and $g': Y \to X'$.
Choose a class $\beta$ on a $D$-trivial coextremal ray.  By Lemma \ref{conepushforward}, there
is some class $\gamma \in \overline{NM}_{1}(Y)$ with $g_{*}\gamma = \beta$.  Define $\alpha = g'_{*}\gamma$.

By \cite{km98}, Lemma 3.38, there is an effective divisor $E$ such that
\begin{equation*}
g^{*}(K_{X} + \Delta) = g'^{*}(K_{X'} + \phi_{*}\Delta) + E.
\end{equation*}
Since $\gamma \cdot E \geq 0$, we have
\begin{align*}
\left( K_{X'} + \phi_{*}\Delta \right) \cdot \alpha &
= g'^{*} \left( K_{X'} + \phi_{*}\Delta \right) \cdot \gamma \\
& \leq g^{*}(K_{X} + \Delta) \cdot \gamma \\
& \leq (K_{X} + \Delta) \cdot \beta \\
& < 0
\end{align*}
Similarly, since each step of $\phi$ is $D$-non-positive, there is some effective divisor $E'$ such that
\begin{equation*}
g^{*}D = g'^{*}\phi_{*}D + E'.
\end{equation*}
The same argument shows that $\phi_{*}D \cdot \alpha \leq 0$.  However, since $\phi_{*}D$ is pseudo-effective, we must have $\phi_{*}D \cdot \alpha = 0$.
\end{proof}

The next lemma is used to show the finiteness of coextremal rays.  The bigness of $\Delta$ plays a key role.

\begin{lem} \label{finitelymanymodels}
Let $(X,\Delta)$ be a $\mathbb{Q}$-factorial klt pair with $\Delta$ big.  Assume Conjecture \ref{terminationofflips}.  Then there are only finitely many models that can be obtained by a sequence of $(K_{X}+\Delta)$ flips and
divisorial contractions.
\end{lem}

\begin{proof}
By the Cone Theorem, if $\Delta$ is big then there are only finitely many $(K_{X} + \Delta)$-negative
extremal rays.

We construct a tree in the following way.  The bottom node represents the variety $X$.
For each $(K_{X} + \Delta)$ flip or divisorial contraction $\phi: X \dashrightarrow X'$, we add a node
representing $X'$, and connect it to the node for $X$ by an edge representing $\phi$.  After we have
completed the (finitely many) additions of edges around the node $X$, we perform the same process on
each model $X'$ and adjoint divisor $K_{X'} + \phi_{*}\Delta$.  We continue until we have exhausted all
possible models $X'$ that can be obtained by a sequence of flips or divisorial contractions.

Note that for any birational map $\phi$ the divisor $\phi_{*}\Delta$ is still big.
This implies that every model $X'$ only has finitely many $(K_{X'} + \phi_{*}\Delta)$-negative
extremal rays.  In other words, our tree is locally finite at each node.  By K\"onig's Lemma,
if there were infinitely many nodes, there would have to be an infinite branch of the tree,
contradicting Conjecture \ref{terminationofflips}.
\end{proof}

\begin{proof}[Proof of Theorem \ref{bigcase}:]
Since $\Delta$ is big, $\Delta \equiv A + B$ for some ample divisor $A$ and some
effective divisor $B$.  For sufficiently small $\tau$ the pair $(X,(1-\tau)\Delta + \tau B)$ is klt.
We define $\Gamma := (1-\tau)\Delta + \tau B$.

Our next goal is to show that for any bounding divisor $D$, there is a $(K_{X} + \Gamma)$-minimal
model program $\phi: X \dashrightarrow X'$ and a Mori fibration $g: X' \to Z$ such that $\phi_{*}D$
is trivial on the fibers.  By Lemma \ref{boundingform2}, we can write $D = \delta_{D} (K_{X} + \Delta) + N_{D}$
for some nef $N_{D}$ and some $\delta_{D} \geq 0$.  We separate into two cases.

First suppose that $\delta_{D} > 0$.  By rescaling $D$ we can write
\begin{equation*}
D \equiv K_{X} + \Gamma + \left( \tau A + N_{D} \right).
\end{equation*}
Since $\tau A + N_{D}$ is ample, we may replace it by a suitable $\mathbb{R}$-linearly equivalent
divisor $A'$ so that $(X,\Gamma + A')$ is klt.  We run the minimal model program with scaling to
obtain a map $\phi: X \dashrightarrow X'$ and a Mori fibration $g: X' \to Z$ such that
$\phi_{*}D$ is trivial on the fibers.

Now suppose that $\delta_{D} = 0$, so that $D = N_{D}$ is nef.  Suppose that $D$ has
vanishing intersection with an extremal ray that corresponds to a flip or divisorial contraction.
If $\psi: X \dashrightarrow X_{1}$ is this operation, then $\psi_{*}D$ is still nef.  We can repeat
this process inductively; by termination of flips, there is some birational contraction $\phi: X \dashrightarrow X'$
such that $\phi_{*}D$ does not have vanishing intersection with any extremal ray corresponding
to a flip or divisorial contraction.  Since $D$ is a bounding divisor, Lemma
\ref{vanishingintersectionpushforward} shows that $\phi_{*}D$ still has vanishing intersection
with some $(K_{X'} + \phi_{*}\Delta)$-negative class $\alpha \in \overline{NM}_{1}(X')$.
Since $\phi_{*}D$ is nef, this means that it must also have vanishing intersection with
a $(K_{X} + \Delta)$-negative extremal ray.  The contraction of this extremal ray is a
$\phi_{*}D$-trivial Mori fibration $g: X' \to Z$.

To each bounding divisor $D$ we have associated a rational map $\phi: X \dashrightarrow X'$
found by running the $(K_{X} + \Gamma)$-minimal model program.  Lemma \ref{finitelymanymodels} shows
there can only be finitely many such models $X'$.  Furthermore, since $\phi_{*}\Gamma$ is big for
any of these maps, there can be only finitely many Mori fibrations on
each $X'$.  By choosing a sufficiently general curve $C$ in a fiber of each Mori fibration, we find
a finite set of movable curves $\{ C_{i} \}$ such that for any bounding divisor $D$, there is some
$C_{i}$ for which $D \cdot C_{i} = 0$.  A straightforward cone argument finishes the proof.
\end{proof}

\nocite{*}
\bibliographystyle{amsalpha}
\bibliography{nefcone}

\end{document}